\documentclass[12pt]{amsart}
\usepackage{amssymb,latexsym}
\usepackage{enumerate}
\usepackage{marginnote}
\usepackage{color}

\makeatletter
\@namedef{subjclassname@2010}{%
  \textup{2010} Mathematics Subject Classification}
\makeatother

\newtheorem{thm}{Theorem}[section]

\newtheorem{lem}[thm]{Lemma}

\theoremstyle{definition}

\newtheorem{exa}[thm]{Example}

\numberwithin{equation}{section}

\newcommand{\R}{\mathbb{R}}

\newcommand{\X}{\mathcal{X}}
\newcommand{\XM}{\X(M)}

\newcommand{\const}{\operatorname{const}}

\newcommand{\D}{\operatorname{D}}
\newcommand{\id}{\operatorname{I}}
\newcommand{\DD}{\mathcal{D}}

\newcommand{\Span}{\operatorname{span}}

\newcommand{\Cbar}{{\bar{C}}}

\newcommand{\Jt}{\widetilde{J}}
\newcommand{\Ct}{\widetilde{\mathbb{C}}}

\begin{document}

%%%%% To ease editing, for IMPAN journals add:

\baselineskip=17pt

%%%%%%%%%%%

%% In the running head, replace first names by initials
%% and give an abbreviation of the title.

\title[On $3$-dimensional $\Jt$-tangent centro-affine hypersurfaces...]{On $3$-dimensional $\Jt$-tangent centro-affine hypersurfaces and $\Jt$-tangent affine hyperspheres with some null-directions}

\author[Z. Szancer]{Zuzanna Szancer}
\address{Department of Applied Mathematics \\ University of Agriculture in Krakow\\ 253c Balicka Street\\
30-198 Krakow, Poland}
\email{zuzanna.szancer@urk.edu.pl}

\date{}

\begin{abstract}
Let $\Jt$ be the canonical para-complex structure on $\R^4$.
In this paper we study $3$-dimensional centro-affine hypersurfaces with a $\Jt$-tangent centro-affine vector field
(sometimes called $\Jt$-tangent centro-affine hypersurfaces) as well as $3$-dimensional $\Jt$-tangent affine hyperspheres with the property that at least one null-direction of the second fundamental form coincides with either $\DD^+$ or $\DD^-$. The main purpose of this paper is to give a full local classification of the above mentioned hypersurfaces. In particular, we prove that every nondegenerate centro-affine hypersurface of dimension $3$ with a $\Jt$-tangent centro-affine vector field which has two null-directions $\DD^+$ and $\DD^-$ must be both an affine hypersphere and a hyperquadric. Some examples of these hypersurfaces are also given.
\end{abstract}



\maketitle

\section{Introduction}
\par Para-complex and paracontact geometry plays an important role in mathematical physics (see \cite{Z}, \cite{CKM}, \cite{KM}). On the other hand affine differential geometry and in particular affine hyperspheres have been extensively studied over past decades. Some relations between para-complex and affine differential geometry can be found in \cite{LS}, \cite{CLS} and \cite{SZ}.
\par Let us denote by $\Ct$ the real algebra of para-complex numbers (for details we refer to \cite{CFG}, \cite{CMMS}) and let $\Jt$ be the canonical para-complex structure on $\R^{2n+2}\cong \Ct^{n+1}$. A transversal vector field for an affine hypersurface $f\colon M\rightarrow \R ^{2n+2}$ is called $\Jt$-tangent if $\Jt$ maps it into a tangent space. Such vector field induces in a natural manner an almost paracontact structure $(\varphi,\xi,\eta)$. We also have the biggest
$\Jt$-invariant distribution in $TM$ which we denote by $\DD$. $\DD$ splits into $\DD^+$ and $\DD^-$-- eigen spaces related to eigen values $+1$ and $-1$ respectively (for details we refer to Sec. 3 of this paper).
\par In \cite{SZ3} the author studied affine hypersurfaces with a  $\Jt$-tangent transversal vector field and gave a local classification of $\Jt$-tangent affine hyperspheres with an involutive distribution $\DD$.
\par In this paper we study real affine hypersurfaces $f\colon M^{3}\rightarrow \R ^{4}\cong \Ct^{2} $ of the para-complex space $\Ct ^{2}$ with a $\Jt$-tangent transversal vector field $C$ and an induced almost paracontact structure $(\varphi,\xi,\eta)$. In \cite{SZ3} the author showed that when $C$ is centro-affine (not necessarily Blaschke) then $f$ can be locally expressed in the form:
\begin{equation}\label{eq::wst}
f(x_1,\ldots,x_{2n},z)=\Jt g(x_1,\ldots,x_{2n})\cosh z-g(x_1,\ldots,x_{2n})\sinh z,
\end{equation}
where $g$ is some smooth immersion defined on an open subset of $\R^{2n}$. Basing on the above result we provide a local classification of all $3$-dimensional nondegenerate centro-affine hypersurfaces with a $\Jt$-tangent transversal centro-affine vector field as well as $\Jt$-tangent affine hyperspheres with the null-direction $\DD^+$ or $\DD^-$. Moreover, in this case distribution $\DD =\DD^+\oplus \DD^-$ is not involutive, so affine hyperspheres are completely different from these studied in \cite{SZ3}.
In particular, we give explicit examples of such hyperspheres.
\par In section 2 we briefly recall the basic formulas of affine differential geometry. We recall the notion of a Blaschke field and an affine hypersphere.
\par In section 3 we recall the definition of an almost paracontact structure introduced for the first time in \cite{KW}. We recall the notion of $\Jt$-tangent transversal vector field and the $\Jt$-invariant distribution. We also recall some elementary results for induced almost paracontact structures that will be used later in this paper ( for details we refer to \cite{SZZ}).
\par In Section 4 we introduce the definition of the null-direction for a nondegenerate affine hypersurface.
In this section we find local representation for nondegenerate centro-affine hyperurfaces $f\colon M\rightarrow \R^4$ with a $\Jt$-tangent centro-affine
transversal vector field with a property that either $\DD^+$ or $\DD^-$ is the null-direction for $f$.
To illustrate this situation we give some explicit examples of such hypersurfaces. Moreover, we prove that every centro-affine nondegenerate hypersurface with a $\Jt$-tangent centro-affine vector field and two null-directions $\DD^+$ and $\DD^-$ is both the affine hypersphere and the hyperquadric.
\par Section 5 concerns the case of $\Jt$-tangent affine hyperspheres. In this section we recall the notion of a $\Jt$-tangent affine hypersphere and prove classification theorems. We show that every $3$-dimensional  $\Jt$-tangent affine hypersphere with the null-direction $\DD^+$ or $\DD^-$ can be locally constructed from two regular flat curves $\alpha, \beta\colon I\rightarrow \R^2$ with a property $\det[\alpha,\alpha']\neq 0$, $\det[\beta,\beta']\neq 0$. As an application we give examples of such hyperspheres.
\section{Preliminaries}
In this section we briefly recall the basic formulas of affine differential
geometry. For more details, we refer to \cite{NS}.
\par Let $f\colon M\rightarrow\R^{n+1}$ be an orientable,
connected differentiable $n$-dimensional hypersurface immersed in
affine space $\R^{n+1}$ equipped with its usual flat connection
$\D$. Then, for any transversal vector field $C$ we have
\begin{equation}\label{eq::FormulaGaussa}
\D_Xf_\ast Y=f_\ast(\nabla_XY)+h(X,Y)C
\end{equation}
and
\begin{equation}\label{eq::FormulaWeingartena}
\D_XC=-f_\ast(SX)+\tau(X)C,
\end{equation}
where $X,Y$ are vector fields tangent to $M$. For any transversal vector field $\nabla$ is a torsion-free connection, $h$ is a symmetric
bilinear form on $M$, called the second fundamental form, $S$ is a tensor of type $(1,1)$, called the shape operator, and $\tau$ is a 1-form, called the transversal connection form.
\par We shall now consider the change of a transversal vector field for a given immersion $f$.
\begin{thm}[\cite{NS}]\label{tw::ChangeOfTransversalField}
Suppose we change a transversal vector field $C$ to
$$
\Cbar=\Phi C+f_\ast(Z),
$$
where $Z$ is a tangent vector field on $M$ and $\Phi$ is a nowhere vanishing function on $M$. Then the affine fundamental form,
the induced connection, the transversal connection form, and the affine shape operator change as follows:
\begin{align*}
& \bar{h}=\frac{1}{\Phi}h,\\
& \bar{\nabla}_XY=\nabla_XY-\frac{1}{\Phi}h(X,Y)Z,\\
& \bar{\tau}=\tau+\frac{1}{\Phi}h(Z,\cdot)+d\ln|\Phi|,\\
& \bar{S}=\Phi S-\nabla_{\cdot}Z+\bar{\tau}(\cdot)Z.
\end{align*}
\end{thm}
\par We have the following
\begin{thm}[\cite{NS}, Fundamental equations]\label{tw::FundamentalEquations}
For an arbitrary transversal vector field $C$ the induced
connection $\nabla$, the second fundamental form $h$, the shape
operator $S$, and the 1-form $\tau$ satisfy
the following equations:
\begin{align}
\label{eq::Gauss}&R(X,Y)Z=h(Y,Z)SX-h(X,Z)SY,\\
\label{eq::Codazzih}&(\nabla_X h)(Y,Z)+\tau(X)h(Y,Z)=(\nabla_Y h)(X,Z)+\tau(Y)h(X,Z),\\
\label{eq::CodazziS}&(\nabla_X S)(Y)-\tau(X)SY=(\nabla_Y S)(X)-\tau(Y)SX,\\
\label{eq::Ricci}&h(X,SY)-h(SX,Y)=2d\tau(X,Y).
\end{align}
\end{thm}
The equations (\ref{eq::Gauss}), (\ref{eq::Codazzih}),
(\ref{eq::CodazziS}), and (\ref{eq::Ricci}) are called the
equation of Gauss, Codazzi for $h$, Codazzi for $S$ and Ricci,
respectively.
\par For a hypersurface immersion $f\colon M\rightarrow \R^{n+1}$
a transversal vector field $C$ is said to be \emph{equiaffine}
(resp. \emph{locally equiaffine}) if $\tau=0$ (resp. $d\tau=0$).
For an affine hypersurface $f\colon M\rightarrow \R^{n+1}$ with the transversal vector field $C$ we consider the following volume element on $M$:
$$
\Theta(X_1,\ldots,X_n):=\det[f_\ast X_1,\ldots,f_\ast X_n,C]
$$ for all $X_1,\ldots ,X_n\in \X(M) $. We call $\Theta $ \emph{the induced volume element} on $M$.
\par Immersion $f\colon M\rightarrow \mathbb{R}^{n+1}$ is said to be a \emph{centro-affine hypersurface} if the position vector $x$ (from origin $o$) for each point $x\in M$ is transversal to the tangent plane of $M$ at $x$. In this case $S=\id$ and $\tau=0$.
\par If $h$ is nondegenerate (that is $h$ defines a semi-Rie\-man\-nian metric on $M$), we say that the hypersurface or the
hypersurface immersion is \emph{nondegenerate}. In this paper we always assume that $f$ is nondegenerate.
\par When $f$ is nondegenerate, there exists a canonical transversal vector field $C$ called \emph{the affine normal field} (or \emph{the Blaschke field}).
The affine normal field is uniquely determined up to sign by the following conditions:
\begin{enumerate}
 \item the induced volume form $\Theta$ is $\nabla$-parallel (i.e. $\tau=0$),
 \item the metric volume form $\omega_h$ of $h$ coincides with the induced volume form $\Theta$.
\end{enumerate} Recall that  $\omega_h$ is defined by
$$\omega_h(X_1,\ldots,X_n)=|\det[h(X_i,X_j)]|^{1/2},$$ where $\{X_1,\ldots ,X_n\}$ is any positively oriented basis relative to the induced volume form $\Theta$. The affine immersion $f$ with a Blaschke field $C$ is called \emph{ a Blaschke hypersurface}.
In this case fundamental equations can be rewritten as follows
\begin{thm}[\cite{NS}, Fundamental equations]\label{tw::FundamentalEquationsBlaschke}
For a Blaschke hypersurface $f$, we have the following fundamental equations:
\begin{align*}
&R(X,Y)Z=h(Y,Z)SX-h(X,Z)SY,\\
&(\nabla_X h)(Y,Z)=(\nabla_Y h)(X,Z),\\
&(\nabla_X S)(Y)=(\nabla_Y S)(X),\\
&h(X,SY)=h(SX,Y).
\end{align*}
\end{thm}
A Blaschke hypersurface is called \emph{an affine hypersphere} if $S=\lambda \id$, where $\lambda =\const.$\\
If $\lambda =0$ $f$ is called \emph{an improper affine hypersphere}, if $\lambda \neq 0$ a hypersurface $f$ is called \emph{a proper affine hypersphere.}

\section{Induced almost paracontact structures}
\par A $(2n+1)$-dimensional manifold $M$ is said to have an
\emph{almost paracontact structure} if there exist on $M$ a tensor
field $\varphi$ of type (1,1), a vector field $\xi$ and a 1-form
$\eta$ which satisfy
\begin{align}
\varphi^2(X)&=X-\eta(X)\xi,\\
\eta(\xi)&=1
\end{align}
for every $X\in TM$
and the tensor field $\varphi$ induces an almost para-complex structure on the distribution $\DD=\operatorname{ker}\eta$. That is the eigendistributions $\DD ^{+},\DD ^{-}$ corresponding to the eigenvalues $1,-1$
of $\varphi$ have equal dimension $n$.
\par Let $\dim M=2n+1$ and $f\colon M\rightarrow \R^{2n+2}$ be
a nondegenerate (relative to the second fundamental form) affine hypersurface. We always assume that $\R^{2n+2}$ is endowed with
the standard para-complex structure $\widetilde{J}$ given by
$$
\widetilde{J}(x_1,\ldots,x_{n+1},y_1,\ldots,y_{n+1}):=(y_1,\ldots,y_{n+1},x_1,\ldots,x_{n+1}).
$$
Let $C$ be a transversal vector field on $M$. We say that $C$ is
\emph{$\widetilde{J}$-tangent} if $\widetilde{J}C_x\in f_\ast(T_xM)$ for every $x\in M$.
We also define a distribution $\DD$ on $M$ as the biggest $\widetilde{J}$-invariant distribution on $M$, that is
$$
\DD_x=f_\ast^{-1}(f_\ast(T_xM)\cap \widetilde{J}(f_\ast(T_xM)))
$$
for every $x\in M$. We have that $\dim\DD _x\geq 2n$. If for some $x$ the $\dim\DD _x=2n+1$ then $\DD _x=T_xM$ and it is not possible to find a $\widetilde{J}$-tangent transversal vector field in a neighbourhood of $x$. Since we only study hypersurfaces with a $\widetilde{J}$-tangent transversal vector field, then we always have $\dim\DD=2n$. The distribution $\DD$ is smooth as an intersection of two smooth distributions and because $\dim \DD$ is constant.  A vector field $X$ is called a \emph{$\DD$-field}
if $X_x\in\DD_x$ for every $x\in M$. We use the notation $X\in\DD$ for vectors as well as for $\DD$-fields.

We say that the distribution $\DD$ is nondegenerate if
$h$ is nondegenerate on $\DD$.
\par To simplify the writing,
we will be omitting $f_\ast$ in front of vector fields in most cases.
\par Let $f\colon M\rightarrow \R^{2n+2}$ be a nondegenerate affine
hypersurface with a $\widetilde{J}$-tangent transversal vector field $C$. Then
we can define a vector field $\xi$, a 1-form $\eta$ and a tensor field
$\varphi$ of type (1,1) as follows:
\begin{align}
\label{xi::eq::0}&\xi:=\widetilde{J}C;\\
\label{etanaD::eq::0}&\eta|_\DD=0 \text{ and } \eta(\xi)=1; \\
\label{phi::eq::0}&\varphi|_\DD=\widetilde{J}|_\DD \text{ and } \varphi(\xi)=0.
\end{align}
It is easy to see that $(\varphi,\xi,\eta)$ is an almost paracontact
structure on $M$. This structure is called the \emph{induced almost
paracontact structure}.
%%%%%%%%%%%%%%%%
%%    LEMMA    %
%%%%%%%%%%%%%%%%
\begin{lem}\label{lm::ZmianaStrukturyPrawieparakontaktowej}
Let $C$ be a $\widetilde{J}$-tangent transversal vector field. Then any other $\widetilde{J}$-tangent transversal vector field $\bar{C}$ has a form:
$$
\bar{C}=\phi C+f_\ast Z,
$$
where $\phi\neq 0$ and $Z\in\DD$. Moreover, if
$(\varphi,\xi,\eta)$ is an almost paracontact structure induced by $C$, then $\bar{C}$ induces an almost paracontact structure
$(\bar\varphi,\bar\xi,\bar\eta)$, where
\begin{align}\label{eq::ChangeOfPhiXiEta}
\begin{cases}
\bar\xi=\phi\xi+\varphi Z,\\
\bar\eta=\frac{1}{\phi}\eta,\\
\bar\varphi=\varphi-\eta(\cdot)\frac{1}{\phi}Z.
\end{cases}
\end{align}
\end{lem}
\begin{proof}
By Theorem \ref{tw::ChangeOfTransversalField} we have that $\bar{C}=\phi C+f_\ast Z$ where $\phi\neq 0$ and $Z\in\XM$.
Since both $C$ and $\bar{C}$ are $\Jt$-tangent then in particular $\Jt f_\ast Z\in f_\ast TM$ that is $Z\in\DD$.
Formulas (\ref{eq::ChangeOfPhiXiEta}) are immediate consequence of (\ref{xi::eq::0})--(\ref{phi::eq::0}).
\end{proof}
For an induced almost paracontact structure we have the following theorem
\begin{thm}[\cite{SZZ}]\label{tw::Wzory}
Let $f\colon M\rightarrow \mathbb{R}^{2n+2}$ be an affine hypersurface with a $\widetilde{J}$-tangent transversal vector field $C$.
If $(\varphi,\xi,\eta)$ is an induced almost paracontact structure on $M$
then the following equations hold:
\begin{align}
\label{Wzory::eq::1}&\eta(\nabla_XY)=h(X,\varphi Y)+X(\eta(Y))+\eta(Y)\tau(X),\\
\label{Wzory::eq::2}&\varphi(\nabla_XY)=\nabla_X\varphi Y-\eta(Y)SX-h(X,Y)\xi,\\
\label{Wzory::eq::3}&\eta([X,Y])=h(X,\varphi Y)-h(Y,\varphi X)+X(\eta(Y))-Y(\eta(X))\\
\nonumber &\qquad\qquad\quad+\eta(Y)\tau(X)-\eta(X)\tau(Y),\\
\label{Wzory::eq::4}&\varphi([X,Y])=\nabla_X\varphi Y-\nabla_Y\varphi X+\eta(X)SY-\eta(Y)SX,\\
\label{Wzory::eq::5}&\eta(\nabla_X\xi)=\tau(X),\\
\label{Wzory::eq::6}&\eta(SX)=-h(X,\xi)
\end{align}
for every $X,Y\in \X(M)$.
\end{thm}

\par We conclude this section with the following useful lemma related to differential equations
\begin{lem}(\cite{SZ3})\label{lm::Differential-Equation}
Let $F\colon I\rightarrow\R^{2n}$ be a smooth function on an interval $I$.
If $F$ satisfies the differential equation
\begin{equation}\label{eq::RR}
F'(z)=-\widetilde{J}F(z),
\end{equation}
then $F$ is of the form
\begin{equation}\label{eq::RR-solution}
F(z)=\widetilde{J}v\cosh z-v\sinh z,
\end{equation}
where $v\in\R^{2n}$.
\end{lem}

\section{Centro-affine hypersurfaces with a $\Jt$-tangent centro-affine vector field}

Let $f\colon M\rightarrow \R^{n+1}$ be a nondegenerate affine hypersurface. A $1$-dimensional, smooth distribution $\mathcal{A}\subset TM$ we call \emph{a null-direction for $f$} if for every $x\in M$ and for every $v\in \mathcal{A}_x$ we have $h_x(v,v)=0$.

%%%%%%%%%%%%%%%%
%%    LEMMA    %
%%%%%%%%%%%%%%%%
\begin{lem}\label{lm::NullDirDplus}
 Let $f\colon M\mapsto \R^4$ be a nondegenerate centro-affine hypersurface with a $\Jt$-tangent centro-affine vector field $C$. If $\DD ^+$ is the null-direction for $f$, then for every point $p\in M$ there exist an open neighbourhood $U$ of $p$ and a vector field $X\in \DD^+$, $X\neq 0$ defined on $U$ such that
\begin{align}
\label{eq::NullDirDplus::1}h(X,X)=0,\\
\label{eq::NullDirDplus::2}\nabla _X\xi=\nabla _{\xi}X=-X,\\
\label{eq::NullDirDplus::3}\nabla _XX=0.
\end{align}
\end{lem}
\begin{proof}
Let $p\in M$ and let $X''$ be a vector field defined on some neighbourhood $U$ of $p$ such that $X''\in \DD^+$, $X''\neq 0$. Since $\DD^+$ is the null-direction for $f$ we have $h(X'',X'')=0$. Now from the formulas (\ref{Wzory::eq::1}), (\ref{Wzory::eq::6}) and due to the fact that  $S=\id$ we get $h(X'',\xi)=0$, $\eta (\nabla_{X''}\xi)=0$, $\eta (\nabla_{\xi}X'')=0$. That is $\nabla_ {X''}\xi, \nabla _{\xi}X''\in\DD$.
Additionally, from (\ref{Wzory::eq::2}) we obtain
$\varphi(\nabla _{X''}\xi)=-X''$
and
$\varphi(\nabla _{\xi}X'')=\nabla _{\xi}\varphi X''=\nabla _{\xi}X''$
since $\varphi X''=X''$. In particular we have $\nabla _{\xi}X''\in\DD^+$.
The above implies that
$$
\nabla _{X''}\xi = -X''
$$
and
$$
\nabla _\xi X'' = \alpha X'',
$$
where $\alpha $ is some smooth function on $U$. We claim that in a neighbourhood of $p$ there exists $X'\in\DD^+,X'\neq 0$ such that
\begin{align}
\label{eq::hXprimXprim} h(X',X')=0,\\
\label{eq::etaXprim}\nabla_{X'}\xi=\nabla_{\xi}X'=-X'.
\end{align}
Indeed, let $X':=\beta X''$, where $\beta \neq 0$, is a solution of the following differential equation $\xi (\beta)=-(\alpha +1)\beta$. Shrinking eventually $U$ if needed, we may assume that $X'$ is defined on $U$. It is easy to verify that $X'$ satisfies (\ref{eq::hXprimXprim}) and (\ref{eq::etaXprim}). By (\ref{Wzory::eq::1}) and (\ref{Wzory::eq::6}) we get $\nabla_{X'}X'=aX'$ for some smooth function $a$. Moreover, from the Gauss equation we have
\begin{align*}
R(\xi,X')X'=0=\nabla_{\xi}(aX')-\nabla_{X'}(-X')=\xi (a)X',
\end{align*} that is $\xi (a)=0$.
Let $X:=bX'$, where $b\neq 0$, is a solution of the following differential equation $X'(b)=-ab$. Again shrinking $U$ if needed we may assume that $X$ is defined on $U$. Since $\xi (a)=0$ we also have that $\xi (b)=0$. Of course we have
\begin{align*}
X\neq 0, X\in \DD^+,\\
h(X,X)=0.
\end{align*}
Now, by straightforward computation we obtain
\begin{align*}
\nabla_XX=0,\quad \nabla_{\xi}X=\nabla_X{\xi}=-X.
\end{align*}
The proof is completed.
\end{proof}

When $\DD^-$ is the null-direction for $f$, one may easily obtain a result similar to Lemma \ref{lm::NullDirDplus}. Namely, we have
%%%%%%%%%%%%%%%%
%%    LEMMA    %
%%%%%%%%%%%%%%%%
\begin{lem}\label{lm::NullDirDminus}
 Let $f\colon M\mapsto \R^4$ be a nondegenerate centro-affine hypersurface with a $\Jt$-tangent centro-affine vector field $C$. If $\DD ^-$ is the null-direction for $f$, then for every point $p\in M$ there exist an open neighbourhood $U$ of $p$ and a vector field $Y\in \DD^-$, $Y\neq 0$ defined on $U$ such that
\begin{align}
\label{eq::NullDirDplus::1}h(Y,Y)=0,\\
\label{eq::NullDirDplus::2}\nabla _Y\xi=\nabla _{\xi}Y=Y,\\
\label{eq::NullDirDplus::3}\nabla _YY=0.
\end{align}
\end{lem}

In order to simplify futher computation, we will need the following technical lemma.
%%%%%%%%%%%%%%%%
%%    LEMMA    %
%%%%%%%%%%%%%%%%
\begin{lem}\label{lm::DetRelation}
Let $g\colon V\rightarrow \R^4$ be a smooth function defined on an open subset $V$ of $\R^2$. Let $I\subset \R$ be an open interval in $\R$ and let function $f\colon V\times I\rightarrow \R^4$ be given by the formula
\begin{align*}
f(x,y,z)=\Jt g(x,y)\cosh z-g(x,y)\sinh z.
 \end{align*}
 Then
 \begin{align*}
 \det [f_x,f_y,f_z,f]=\det [g_x,g_y,g, \Jt g].
 \end{align*}
\end{lem}
\begin{proof}
We have
\begin{align*}
\det [f_x,f_y,f_z,f] &= \det [f_x,f_y,\Jt g\sinh z-g\cosh z,\Jt g\cosh z-g\sinh z] \\
&= (\cosh^2z-\sinh^2z)\det[f_x,f_y,\Jt g,g]=\det[f_x,f_y,\Jt g,g].
\end{align*}
Now
\begin{align*}
\det[f_x,f_y,\Jt g,g] &= \det [\Jt g_x\cosh z-g_x\sinh z,\Jt g_y\cosh z-g_y\sinh z,\Jt g,g] \\
&= \cosh^2z\det[\Jt g_x,\Jt g_y,\Jt g,g]+\sinh^2z\det[g_x,g_y,\Jt g,g]\\
&\phantom{=}-\sinh z \cosh z(\det[\Jt g_x,g_y,\Jt g,g]+\det[g_x,\Jt g_y,\Jt g,g])\\
&=(\cosh^2z-\sinh^2z)\det [g_x,g_y,g, \Jt g],
\end{align*}
where the last equality is a consequence of the following straightforward observation:
\begin{align}
\label{eq::Detv1v2v3}\det [\Jt v_1,v_2,v_3,\Jt v_3]=-\det [v_1,\Jt v_2,v_3,\Jt v_3]
\end{align}
for every $v_1,v_2,v_3\in\R^4$.
Summarising
$$
\det [f_x,f_y,f_z,f] = \det[f_x,f_y,\Jt g,g] = \det [g_x,g_y,g, \Jt g].
$$
\end{proof}
Now we shall state the classification theorems for centro-affine hypersurfaces.
%%%%%%%%%%%%%%%%
%%    THEOREM  %
%%%%%%%%%%%%%%%%
\begin{thm}\label{th::CAHwithDplus}
Let $f\colon M\mapsto \R^4$ be a nondegenerate centro-affine hypersurface with a $\Jt$-tangent centro-affine vector field $C$. If $\DD^+$ is the null-direction for $f$ then for every point $p\in M$ there exists a neighbourhood $U$ of $p$ such that $f|_{U}$ can be expressed in the form:
\begin{align}\label{eq::flocal}
f(x,y,z)=\Jt g(x,y)\cosh z-g(x,y)\sinh z,
\end{align}
where $g(x,y)=x\cdot \gamma _1(y)+\gamma _2(y)$ and $\gamma_1, \gamma_2$ are some curves such that $\Jt \gamma_1=\gamma_1$ and $\det [\gamma_1, \gamma_2',\gamma_2,\Jt \gamma_2]\neq 0$. Moreover, when $\gamma_1, \gamma_2$ are smooth curves such that $\Jt \gamma_1=\gamma_1$ and $\det [\gamma_1, \gamma_2',\gamma_2,\Jt \gamma_2]\neq 0$ then $f$ given by {\rm(\ref{eq::flocal})} is the nondegenerate centro-affine hypersurface with a $\Jt$-tangent centro-affine vector field with the null-direction $\DD^+$.
\end{thm}
\begin{proof}
By Lemma \ref{lm::NullDirDplus} for every $p\in M$ there exist a neighbourhood $U$ and vector field $X\in\DD^+, X\neq 0$ defined on this neighbourhood such that
\begin{align*}
h(X,X)=0,\\
\nabla_XX=0,
\end{align*}
\begin{align}\label{eq::Nabla_XXi}
\nabla_X{\xi}=\nabla_{\xi}X=-X.
\end{align}
By (\ref{eq::Nabla_XXi}) we have $[X,\xi]=0$,
so there exists a local coordinate system $(x,y,z)$ around $p$ such that
$$
\frac{\partial}{\partial x}=X\quad \text{and}\quad \frac{\partial}{\partial z}=\xi
$$
in some neighbourhood of $p$. Without loss of generality  we may assume that the system $(x,y,z)$ is defined on $U$.
Since $\Jt C=f_{\ast}(\xi)$ we have
\begin{align}\label{eq::fzet}
f_z=f_{\ast}(\frac{\partial}{\partial z})=f_{\ast}(\xi)=\Jt C=-\Jt f.
\end{align}
By the Gauss formula we also have
\begin{align}
\label{eq::fxx} f_{xx}=0,\\
\label{eq::fxzet} f_{xz}=f_{\ast}(\nabla _{\frac{\partial}{\partial x}}{\frac{\partial}{\partial z}})=f_{\ast}(\nabla_{\frac{\partial}{\partial x}}{\xi})=f_{\ast}(- \frac{\partial}{\partial x})=-f_x.
\end{align}
Solving (\ref{eq::fzet}) (see Lemma \ref{lm::Differential-Equation}) we get that $f$ can be locally expressed in the form:
\begin{align*}
f(x,y,z)=\Jt g(x,y)\cosh z-g(x,y)\sinh z
\end{align*}
where $g\colon V\ni (x,y)\mapsto g(x,y)\in \R^4$ is some smooth function defined on an open subset $V$ of $\R^2$.
Moreover, by (\ref{eq::fxx}) and (\ref{eq::fxzet}) we have
\begin{align}
\label{eq::gxx=0} g_{xx}=0,\\
\label{eq::Jgx=gx} \Jt g_x=g_x.
\end{align}
Solving (\ref{eq::gxx=0}) we obtain that
\begin{align*}
g(x,y)=x\cdot \gamma _1(y)+\gamma _2(y),
\end{align*}
where $\gamma_1, \gamma_2$ are some smooth curves. From (\ref{eq::Jgx=gx}) we get that $\Jt \gamma_1=\gamma_1$. Since $f$ is a centro-affine  immersion we have $\det[f_x,f_y,f_z,f]\neq 0$ and in consequence $\det [\gamma_1, \gamma_2',\gamma_2,\Jt \gamma_2]\neq 0$, because by (\ref{eq::Detv1v2v3}) $\det [\gamma_1, \gamma_1',\gamma_2,\Jt \gamma_2]= 0$.
\par In order to prove the last part of the theorem, first note that $\det [\gamma_1, \gamma_2,\gamma_2',\Jt \gamma_2]\neq 0$ implies that $f$ given by (\ref{eq::flocal}) is a centro-affine immersion. Indeed, by Lemma \ref{lm::DetRelation} and due to the fact that $\Jt \gamma_1=\gamma_1$ we have
\begin{align*}
\det [f_x,f_y,f_z,f]&=\det [\gamma_1,x\gamma_1'+\gamma_2',x\gamma_1+\gamma_2,x\gamma_1+\Jt\gamma_2]\\
&=\det [\gamma_1,x\gamma_1'+\gamma_2',\gamma_2,\Jt\gamma_2]\\
&=\det [\gamma_1,\gamma_2',\gamma_2,\Jt\gamma_2]\neq 0.
\end{align*}
Moreover, since $\Jt f=-f_z$, the centro-affine transversal vector field is $\Jt$-tangent. Now it is enough to show that $\DD^+$ is the null-direction for $f$. Since
\begin{align}\label{eq::f_x}
f_x=\Jt\gamma_1(y)\cosh z-\gamma_1 (y)\sinh z=\gamma_1(y)(\cosh z-\sinh z)
\end{align}
we have that $\Jt f_x=f_x$ that is $\DD^+=\Span \{\frac{\partial}{\partial x}\}$. From (\ref{eq::f_x}) we also have that $f_{xx}=0$ so, in particular, $h(\frac{\partial}{\partial x},\frac{\partial}{\partial x})=0$ and $\DD^+$ is the null-direction for $f$.
\end{proof}

When $\DD^-$ is the null-direction for $f$, using Lemma \ref{lm::NullDirDminus}, one may prove an analogous result to Theorem \ref{th::CAHwithDplus}.
That is we have:
\begin{thm}\label{th::CAHwithDminus}
Let $f\colon M\mapsto \R^4$ be a nondegenerate centro-affine hypersurface with a $\Jt$-tangent centro-affine vector field $C$. If $\DD^-$ is the null-direction for $f$ then for every point $p\in M$ there exists a neighbourhood $U$ of $p$ such that $f|_{U}$ can be expressed in the form:
\begin{align}\label{eq::flocalDminus}
f(x,y,z)=\Jt g(x,y)\cosh z-g(x,y)\sinh z,
\end{align}
where $g(x,y)=y\cdot \gamma _1(x)+\gamma _2(x)$ and $\gamma_1, \gamma_2$ are some curves such that $\Jt \gamma_1=-\gamma_1$ and $\det [\gamma_1, \gamma_2',\gamma_2,\Jt \gamma_2]\neq 0$. Moreover, when $\gamma_1, \gamma_2$ are smooth curves such that  $\Jt \gamma_1=-\gamma_1$ and $\det [\gamma_1, \gamma_2',\gamma_2,\Jt \gamma_2]\neq 0$ then $f$ given by {\rm(\ref{eq::flocalDminus})} is the nondegenerate centro-affine hypersurface with the $\Jt$-tangent centro-affine vector field with the null-direction $\DD^-$.
\end{thm}
%\begin{proof}
%\end{proof}
In order to illustrate the above theorems we give some explicit examples of centro-affine hypersurfaces with a $\Jt$-tangent centro-affine vector field
with the null-direction $\DD^+$ or $\DD^-$.

%%%%%%%%%%%%%%%%
%%    EXAMPLE  %
%%%%%%%%%%%%%%%%
\begin{exa}\label{ex::CentroNullDplus}
Let us consider the affine immersion defined by $$f\colon \R^3\ni (x,y,z)\mapsto \Jt g(x,y)\cosh z-g(x,y)\sinh z\in\R^4$$ where
\begin{align*}
g\colon \R^2\ni (x,y)\mapsto \left [\begin{matrix}
xy+1\\
-x+y\\
xy\\
-x
\end{matrix}\right ]\in \R^4
\end{align*}
with the transversal vector field $C=-f$. Of course $C$ is $\Jt$-tangent, $\tau =0$ and $S=\id$.
By straightforward computations we obtain
\begin{align*}
h= \left [\begin{matrix}
0 & -\frac{1}{y^2+1} & 0\\
\frac{1}{y^2+1} & 0 & \frac{2x+y}{y^2+1}\\
0 & \frac{2x+y}{y^2+1} &-1
\end{matrix}\right ]
\end{align*}
in the canonical basis $\{\frac{\partial}{\partial x},\frac{\partial}{\partial y},\frac{\partial}{\partial z}\}$ of $\R^3$ so in particular $f$ is nondegenerate.
It is easy to see that $\frac{\partial}{\partial x}\in\DD^+$ and since $h(\frac{\partial}{\partial x},\frac{\partial}{\partial x})=0$ distribution $\DD^+$
is the null-direction for $f$. One may also compute
\begin{align*}
\Theta\Big(\frac{\partial}{\partial x},\frac{\partial}{\partial y},\frac{\partial}{\partial z}\Big)=y^2+1\\
\omega_h\Big(\frac{\partial}{\partial x},\frac{\partial}{\partial y},\frac{\partial}{\partial z}\Big)=\frac{1}{y^2+1}
\end{align*}
so $f$ is not an affine hypersphere. Moreover, $f$ is not a hyperquadric since $\nabla h\neq 0$.
\end{exa}
One may construct a similar example, when $\DD^-$ is the null-direction. Namely we have:
%%%%%%%%%%%%%%%%
%%    EXAMPLE  %
%%%%%%%%%%%%%%%%
\begin{exa}\label{ex::CentroNullDminus}
Let us consider the affine immersion defined by $$f\colon \R^3\ni (x,y,z)\mapsto \Jt g(x,y)\cosh z-g(x,y)\sinh z\in\R^4,$$ where
\begin{align*}
g\colon \R^2\ni (x,y)\mapsto \left [\begin{matrix}
xy+1\\
x-y\\
-xy\\
y
\end{matrix}\right ]\in \R^4
\end{align*}
with the transversal vector field $C=-f$. Of course $C$ is $\Jt$-tangent, $\tau =0$ and $S=\id$.
By straightforward computations we obtain
\begin{align*}
h= \left [\begin{matrix}
0 & -\frac{1}{x^2+1} & -\frac{x+2y}{x^2+1}\\
\frac{1}{x^2+1} & 0 & 0\\
-\frac{x+2y}{x^2+1} & 0 &-1
\end{matrix}\right ]
\end{align*}
in the canonical basis $\{\frac{\partial}{\partial x},\frac{\partial}{\partial y},\frac{\partial}{\partial z}\}$ of $\R^3$ so in particular $f$ is nondegenerate.
It is easy to see that $\frac{\partial}{\partial y}\in\DD^-$ and since $h(\frac{\partial}{\partial y},\frac{\partial}{\partial y})=0$ distribution $\DD^-$
is the null-direction for $f$. One may also compute
\begin{align*}
\Theta\Big(\frac{\partial}{\partial x},\frac{\partial}{\partial y},\frac{\partial}{\partial z}\Big)=x^2+1\\
\omega_h\Big(\frac{\partial}{\partial x},\frac{\partial}{\partial y},\frac{\partial}{\partial z}\Big)=\frac{1}{x^2+1}
\end{align*}
so $f$ is not an affine hypersphere. Moreover, $f$ is not a hyperquadric since $\nabla h\neq 0$.
\end{exa}
The next example have the property that both $\DD^+$ and $\DD^-$ are null-directions.
%%%%%%%%%%%%%%%%
%%    EXAMPLE   %
%%%%%%%%%%%%%%%%
\begin{exa}\label{ex::CentroNullDplusDminusHip1}
Let us consider a function $f$ defined by $$f\colon \R^3 \ni (x,y,z)\mapsto \Jt g(x,y)\cosh z-g(x,y)\sinh z\in\R^4$$ where

\begingroup
\renewcommand*{\arraystretch}{1.5}
\begin{align*}
g\colon \R^2\ni (x,y)\mapsto \left [\begin{matrix}
x+\frac{1}{2}y-xy+\frac{1}{4}\\
x+\frac{1}{2}y+xy-\frac{1}{4}\\
x-\frac{1}{2}y+xy+\frac{3}{4}\\
x-\frac{1}{2}y-xy-\frac{3}{4}
\end{matrix}\right ]\in \R^4.
\end{align*}
\endgroup

It is easy to verify that $f$ is an immersion and $C=-f$ is a transversal vector field. Of course $C$ is $\Jt$-tangent,  $\tau =0$ and $S=\id$.
By straightforward computations we obtain
\begin{align*}
h= \left [\begin{matrix}
0 & -2 & -4y\\
-2 & 0 & 0\\
-4y & 0 &-1
\end{matrix}\right ]
\end{align*}
in the canonical basis $\{\frac{\partial}{\partial x},\frac{\partial}{\partial y},\frac{\partial}{\partial z}\}$.
We also compute
\begin{align*}
\Theta\Big(\frac{\partial}{\partial x},\frac{\partial}{\partial y},\frac{\partial}{\partial z}\Big)=2\\
\omega_h\Big(\frac{\partial}{\partial x},\frac{\partial}{\partial y},\frac{\partial}{\partial z}\Big)=2
\end{align*}
so $f$ is nondegenerate and $f$ is an affine hypersphere. It is easy to see that $\frac{\partial}{\partial y}\in \DD^-$ and since $h(\frac{\partial}{\partial y},\frac{\partial}{\partial y})=0$ the distribution $\DD^-$ is the null-direction for $f$.
\par Let
$$
X:=\frac{1}{4}\cdot \frac{\partial}{\partial x}+y^2\cdot \frac{\partial}{\partial y}-y\cdot \frac{\partial}{\partial z}.
$$
By straightforward computations it can be checked that $X\in\DD^+$ and $h(X,X)=0$ so $\DD^+$ is also  the null-direction for $f$.
\par Moreover, one may compute that $\nabla h=0$ so $f$ is a hyperquadric.
\end{exa}

The first two examples are neither hyperquadrics nor affine hyperspheres, however the last example is both the affine hypersphere and the hyperquadric.
This is not a coincidence. Namely, we have the following theorem:
%%%%%%%%%%%%%%%%
%%    THEOREM  %
%%%%%%%%%%%%%%%%
\begin{thm}\label{thm::CentroWithDplusAndDminus}
Let $f\colon M\mapsto \R^4$ be a centro-affine nondegenerate  hypersurface with a $\Jt$-tangent centro-affine vector field $C$ and let $\DD^-$ and $\DD^+$ are null-directions for $f$. Then
$f$ is both the affine hypersphere and the hyperquadric. Moreover, if $f\colon M\rightarrow \R^4$ is a centro-affine nondegenerate hyperquadric with a $\Jt$-tangent centro-affine vector field $C$ then $\DD^-$ and $\DD^+$ are null-directions for $f$.
\end{thm}
In order to prove the above theorem we need the following lemma:

%%%%%%%%%%%%%%%%
%%    LEMMA    %
%%%%%%%%%%%%%%%%
\begin{lem}\label{lm::NullDirDplusDminus}
Let $f\colon M\mapsto \R^4$ be a nondegenerate centro-affine hypersurface with a $\Jt$-tangent centro-affine vector field $C$. If both $\DD^+$ and $\DD^-$ are null-directions for $f$, then for every point $p\in M$ there exist a neighbourhood $U$ of $p$ and vector fields $X,Y$ defined on $U$ such that $X\neq 0, X\in \DD^+$, $Y\neq 0, Y\in \DD^-$ and the following conditions are satisfied:
\begin{align}
\label{eq::NullDirDplusDminus::1}\nabla _X\xi=\nabla _{\xi}X=-X, \quad \nabla _Y\xi=\nabla _{\xi}Y=Y,\\
\label{eq::NullDirDplusDminus::2}h(X,X)=0,\quad h(Y,Y)=0,\\
\label{eq::NullDirDplusDminus::3}\nabla _XX=0,\\
\label{eq::NullDirDplusDminus::4}\nabla _YY=0,\\
\label{eq::NullDirDplusDminus::5}\nabla _XY=aY-e^{\beta}\xi,\\
\label{eq::NullDirDplusDminus::6}\nabla _YX=bX+e^{\beta}\xi,\\
\label{eq::NullDirDplusDminus::7} h(X,Y)=e^{\beta},\\
\label{eq::NullDirDplusDminus::8} \xi (\beta)=\xi (a)=\xi (b)=0,\\
\label{eq::NullDirDplusDminus::9} X(\beta)=a, Y(\beta)=b,
\end{align}
for some smooth functions $a, b, \beta$ on $U$.
\end{lem}
\begin{proof}
By Lemma \ref{lm::NullDirDplus} and Lemma \ref{lm::NullDirDminus} there exist $X\in \DD^+$, $X\neq 0$, $Y\in \DD^-$, $Y\neq 0$ defined on some neighbourhood of $p$ such that (\ref{eq::NullDirDplusDminus::1})--(\ref{eq::NullDirDplusDminus::4}) hold.
Since $f$ is nondegenerate $h(X,Y)\neq 0$. Without loss of generality (replacing $Y$ by $-Y$ if needed) we may assume that $h(X,Y)=e^{\beta}$, where $\beta$ is a smooth function on $U$.
From formulas (\ref{Wzory::eq::1}), (\ref{Wzory::eq::2}), (\ref{Wzory::eq::6}) and due to the fact that  $S=\id$ we get
\begin{align*}
\eta (\nabla _XY)&=-h(X,Y)=-e^{\beta},\\
\eta (\nabla _YX)&=h(X,Y)=e^{\beta},\\
\varphi (\nabla _XY)&=-\nabla _XY-e^{\beta}\xi,\\
\varphi (\nabla _YX)&=\nabla _YX-e^{\beta}\xi.
\end{align*}
The above implies that there exist smooth functions $a$ and $b$ such that
\begin{align*}
\nabla _XY=aY-e^{\beta}\xi,\\
\nabla _YX=bX+e^{\beta}\xi.
\end{align*}
The above and the Gauss equation imply
\begin{align*}
e^{\beta}\xi=R(\xi ,X)Y=\xi (a)Y-e^{\beta}\xi (\beta)\xi +e^{\beta}\xi
\end{align*}
and
\begin{align*}
e^{\beta}\xi=R(\xi ,Y)X=\xi (b)X+e^{\beta}\xi (\beta)\xi +e^{\beta}\xi.
\end{align*}
The above implies $\xi (\beta)=\xi (a)=\xi (b)=0.$
From the Codazzi equation for $h$ we have
\begin{align*}
0=(\nabla_Yh)(X,X)=(\nabla_Xh)(X,Y)&=X(e^{\beta})-h(X,aY-e^{\beta}\xi)\\
&=e^{\beta}X(\beta)-ae^{\beta}
\end{align*} that is $X(\beta)=a$.
In a similar way from the Codazzi equation for $h$ we get
\begin{align*}
0=(\nabla_Xh)(Y,Y)=(\nabla_Yh)(X,Y)&=Y(e^{\beta})-h(bX+e^{\beta}\xi,Y)\\
&=e^{\beta}Y(\beta)-be^{\beta}
\end{align*} that is $Y(\beta)=b$.
\end{proof}

%%%%%%%%%%%%%%%%%%%%%%%%%%%%%%
%%%%%%%%%%%%%%%%%%%%%%%%%%%%%%
\begin{proof}[Proof of Theorem \ref{thm::CentroWithDplusAndDminus}]
Let $X$, $Y$ and $\xi$ be as in Lemma \ref{lm::NullDirDplusDminus}. It is easy to see that
\begin{align*}
(\nabla _{\xi}h) (U,V)=\xi (h(U,V))-h(\nabla_{\xi}U,V)-h(U,\nabla _{\xi}V)=0
\end{align*} for $U,V\in \{X,Y,\xi \}$, because $h(U,V)\in \{0,-1,e^{\beta}\}$ and $\xi (\beta)=0$. So we have $\xi (h(U,V))=0$.
Moreover we have
\begin{align*}
(\nabla_{U}h)(V,W)=0
\end{align*}
for every $U,V,W\in \{X,Y\}$.
The above implies that
\begin{align*}
(\nabla _{Z_1}h)(Z_2,Z_3)=0
\end{align*}
for every $Z_1,Z_2,Z_3\in \XM $.\\
%To get that $\nabla \omega_h=0$ we compute
%\begin{align*}
%\omega_h(X,Y,\xi) &=\sqrt {\Bigg|\det
%\left [\begin{matrix}
%h(X,X) & h(X,Y) & h(X,\xi)\\
%h(Y,X) & h(Y,Y) & h(Y,\xi)\\
%h(\xi ,X) & h(\xi ,Y) & h(\xi ,\xi)
%\end{matrix}\right ]\Bigg|}\\
%&=\sqrt {\Bigg|\det
%\left [\begin{matrix}
%0 & h(X,Y) & 0\\
%h(Y,X) & 0 & 0\\
%0 & 0 & -1
%\end{matrix}\right ]\Bigg|}\\
%&=\sqrt{(h(X,Y))^2}=e^{\beta}
%\end{align*}
%Now we have
%\begin{align*}
%(\nabla _X\omega_h)(X,Y,\xi)&=X(\omega_h(X,Y,\xi))-\omega_h(\nabla_XX,Y,\xi)-\omega_h(X,\nabla _XY,\xi)\\
%&\phantom{=}-\omega _h(X,Y,\nabla _X{\xi})\\
%&=X(e^{\beta})-\omega_h(X,aY-e^{\beta}\xi,\xi)=e^{\beta}X(\beta)-ae^{\beta}=0
%\end{align*}
%and
%\begin{align*}
%(\nabla _Y\omega_h)(X,Y,\xi)&=Y(\omega_h(X,Y,\xi))-\omega_h(\nabla_YX,Y,\xi)-\omega_h(X,\nabla _YY,\xi)\\
%&\phantom{=}-\omega _h(X,Y,\nabla _Y{\xi})\\
%&=Y(e^{\beta})-\omega_h(bX+e^{\beta}\xi,Y,\xi)=e^{\beta}Y(\beta)-be^{\beta}=0.
%\end{align*}
%We also have
%\begin{align*}
%(\nabla _{\xi}\omega_h)(X,Y,\xi)=\xi (e^{\beta})-\omega_h(-X,Y,\xi)-\omega_h(X,Y,\xi)=0,
%\end{align*} since $\xi (\beta)=0$. Summarizing $\nabla \omega_h=0$ and the proof is completed.
Since $\nabla h=0$ then in particular $\nabla \omega_h=0$ so $f$ is an  affine hypersphere.
\par In order to prove the second part of the theorem first note that since $f$ is a centro-affine hyperquadric we have $S=\id$, $\tau=0$ and
 $\nabla h=0$. Let $X\in \DD$, then we have
 $$
0=(\nabla_X h)(X,\xi)=-h(\nabla_XX,\xi)-h(X,\nabla_X\xi).
 $$
By Theorem \ref{tw::Wzory} we also have
$$
-h(\nabla_XX,\xi)=\eta(\nabla_XX)=h(X,\varphi X)
$$
and
$$
-h(X,\nabla_X\xi)=-h(X,-\varphi X)=h(X,\varphi X).
$$
Summarising
$$
h(X,\varphi X)=0
$$
for every $X\in\DD$.
The above immediately implies that $h(X,X)=0$ if $X\in\DD^+$ or $X\in\DD^-$ so both $\DD^+$ and $\DD^-$ are null-directions for $f$.
\end{proof}

\section{$\Jt$-tangent affine hyperspheres}
It is well known that every proper affine hypersphere is in particular a centro-affine hypersurface. Of course the converse is not true in general.
The purpose of this section is to give a local classification of $3$-dimensional $\Jt$-tangent affine hyperspheres $f$ with the property that either $\DD^+$ or $\DD^-$
is the null-direction for $f$.

\par First recall (\cite{SZ3}) that an affine hypersphere with a transversal $\Jt$-tangent Blaschke field is called \emph{a $\Jt$-tangent affine hypersphere}. Moreover, in \cite{SZ3} we proved that every $\Jt$-tangent affine hypersphere is proper. Now, we shall prove classification theorems for affine hyperspheres with null-directions.
%%%%%%%%%%%%%%%%
%%    THEOREM  %
%%%%%%%%%%%%%%%%
\begin{thm}
Let $f\colon M\mapsto \R^4$ be an affine hypersphere with the null-direction $\DD^+$ and the $\Jt$-tangent Blaschke field $C$. Then there exist smooth planar curves
$\alpha\colon I\rightarrow \R^2$, $\beta\colon I\rightarrow \R^2$ with the property
\begin{align}\label{eq::detalphabeta}
\det [\alpha,\alpha']\cdot \det [\beta,\beta']\neq 0
\end{align}
and a smooth function $A\colon I\rightarrow \R$, where $I$ is an open interval in $\R$, such that $f$ can be locally expressed in the form
\begin{align}\label{eq::formf}
f(x,y,z)=\Jt g(x,y)\cosh z-g(x,y)\sinh z,
\end{align}
where $g(x,y)=(x+A(y))\cdot (\alpha,\alpha)(y)+B(y)(\alpha',\alpha')(y)+(\beta,-\beta)(y)$ and
$$
B(y)=\frac{E}{\sqrt{|\det [\alpha(y),\alpha'(y)]\cdot \det [\beta(y),\beta'(y)]|}}
$$
for some non-zero constant $E$.
\par Moreover, for any smooth curves $\alpha, \beta\colon I\rightarrow \R^2$ such that (\ref{eq::detalphabeta}) and for any smooth function $A\colon I\rightarrow \R$ and a non-zero constant $E$ the function $f$ given by {\rm(\ref{eq::formf})} is the $\Jt$-tangent affine hypersphere with the null-direction $\DD^+$.
\end{thm}

%\begin{thm}
%Let $f\colon M\mapsto \R^4$ be an affine hypersphere with a null-direction $\DD^+$ and $\Jt$-tangent Blaschke field $C$. Then there exist smooth curves
%$\gamma_1=(P_1,P_2,P_1,P_2)$ (so $\Jt \gamma_1=\gamma_1$) and $\gamma_2=(Q_1,Q_2,-Q_1,-Q_2)$ (so $\Jt \gamma_2=-\gamma_2$) with the property $\det [\gamma_1, \gamma_1',\gamma_2,\gamma_2']\neq 0$
%and a smooth function $A$ such that $f$ can be locally expressed in the form
%\begin{align*}
%f(x,y,z)=\Jt g(x,y)\cosh z-g(x,y)\sinh z,
%\end{align*}
%where $g(x,y)=(x+A(y))\cdot \gamma _1(y)+B(y)\gamma _1'(y)+\gamma _2(y)$ and
%$$
%B(y)=\frac{E}{\sqrt{|(P_1'(y)P_2(y)-P_1(y)P_2'(y))(Q_1'(y)Q_2(y)-Q_1(y)Q_2'(y))|}}
%$$
%for some non-zero constant $E$.
%\end{thm}

\begin{proof}
Since $f$ is a $\Jt$-tangent affine hypersphere it is proper so there exists $\lambda\neq 0$ such that $C=-\lambda f$. In particular $f$ is centro-affine. Now, by Theorem \ref{th::CAHwithDplus} for every point $p\in M$ there exists a neighborhood $U$ of $p$ such that $f|_{U}$ has a form:
\begin{align*}
f(x,y,z)=\Jt g(x,y)\cosh z-g\sinh z,
\end{align*}
where $g(x,y)=x\cdot \gamma _1(y)+\widetilde{\gamma_2}(y)$, $\Jt \gamma_1=\gamma_1$ and $\det [\gamma_1, \widetilde{\gamma_2}',\widetilde{\gamma_2},\Jt \widetilde{\gamma_2}]\neq 0.$ Let us define $\gamma_2:=\frac{1}{2}\cdot (\widetilde{\gamma_2}-\Jt \widetilde{\gamma_2})$ and $\gamma_3:=\frac{1}{2}\cdot (\widetilde{\gamma_2}+\Jt \widetilde{\gamma_2})$. Of course we have
\begin{align*}
\Jt \gamma_2=-\gamma_2,\quad \Jt\gamma_3=\gamma_3 \quad \text{and}\quad \widetilde{\gamma_2}=\gamma_2+\gamma_3.
 \end{align*}Because $h(\frac{\partial}{\partial x},\frac{\partial}{\partial x})=0$, $h(\frac{\partial}{\partial x},\frac{\partial}{\partial z})=0$ and since $h$ is nondegenerate we have that $h(\frac{\partial}{\partial x},\frac{\partial}{\partial y})\neq 0$. By the Gauss formula we have
\begin{align*}
f_{xy}=f_{\ast}(\nabla_{\frac{\partial}{\partial x}}\frac{\partial}{\partial y})-\lambda h(\frac{\partial}{\partial x},\frac{\partial}{\partial y})f.
\end{align*}
Since $h(\frac{\partial}{\partial x},\frac{\partial}{\partial y})\neq 0$ the above implies that $f_{xy}$ and $f_x$ are linearly independent and as a consequence $\gamma_1$ and $\gamma_1'$ are linearly independent too. Moreover $\Jt \gamma_1'=\gamma_1'$ that is $\gamma_1, \gamma_1'$ form a basis of the eigenspace of $\Jt$ related to the eigenvalue $1$. It means that there exist smooth functions $A,B$ such that
\begin{align*}
\gamma_3=A\gamma_1+B\gamma_1'.
\end{align*}
Now, $g$ can be rewritten in the form
\begin{align*}
g(x,y)=(x+A(y))\gamma_1(y)+B(y)\gamma_1'(y)+\gamma_2(y).
\end{align*}
Since $\det [\gamma_1, \widetilde{\gamma_2},\widetilde{\gamma_2}',\Jt \widetilde{\gamma_2}]\neq 0$ we compute
\begin{align*}
0\neq \det [\gamma_1,\gamma_2 +\gamma_3,\gamma_2' +\gamma_3',-\gamma_2 +\gamma_3]&=\det [\gamma_1,\gamma_2,\gamma_2' +\gamma_3',-\gamma_2 +\gamma_3]\\
&\phantom{=}+\det [\gamma_1,\gamma_3,\gamma_2' +\gamma_3',-\gamma_2 +\gamma_3]\\
&=\det [\gamma_1,\gamma_2,\gamma_2',\gamma_3]+\det [\gamma_1,\gamma_2,\gamma_3',\gamma_3]\\
&\phantom{=}-\det [\gamma_1,\gamma_3,\gamma_2',\gamma_2]-\det [\gamma_1,\gamma_3,\gamma_3',\gamma_2]\\
&=2\det [\gamma_1,\gamma_2,\gamma_2',\gamma_3]=2B\det [\gamma_1,\gamma_2,\gamma_2',\gamma_1'].
\end{align*}
So $\det [\gamma_1,\gamma_1',\gamma_2,\gamma_2']\neq 0.$
Now, since $\Jt \gamma_1=\gamma_1$ and $\Jt \gamma_2=-\gamma_2$ there exist smooth curves $\alpha =(\alpha_1,\alpha_2)\colon I\rightarrow \R^2$, $\beta=(\beta_1,\beta_2)\colon I\rightarrow \R^2$ such that
$$
\gamma_1=(\alpha,\alpha)=\left [\begin{matrix}
\alpha_1\\
\alpha_2\\
\alpha_1\\
\alpha_2
\end{matrix}\right ]\quad\text{and}\quad
\gamma_2=(\beta ,-\beta)=\left [\begin{matrix}
\beta_1\\
\beta_2\\
-\beta_1\\
-\beta_2
\end{matrix}\right ].
$$
By straightforward computation we get
\begin{align*}
\det [\gamma_1,\gamma_1',\gamma_2,\gamma_2']&=4\cdot (\alpha_1\alpha_2'-\alpha_1'\alpha_2)\cdot (\beta_1\beta_2'-\beta_1'\beta_2)\\
&=4\det [\alpha,\alpha']\cdot\det [\beta, \beta'],\\
\Theta (\frac{\partial}{\partial x},\frac{\partial}{\partial y},\frac{\partial}{\partial z}) &=2\lambda B\cdot \det [\gamma_1,\gamma_1',\gamma_2,\gamma_2']\\
&=8\lambda \cdot B\cdot\det [\alpha,\alpha']\cdot\det [\beta, \beta'],\\
\omega _h (\frac{\partial}{\partial x},\frac{\partial}{\partial y},\frac{\partial}{\partial z})&=\frac{1}{2|\lambda|\sqrt{|\lambda|}\cdot |B|}.
\end{align*}
Since $f$ is an affine hypersphere we have $\omega_h=|\Theta |$ that is
\begin{align*}
|B|=\frac{1}{4|\lambda|\sqrt[4]{|\lambda|}\sqrt{|\det [\alpha,\alpha']\cdot\det [\beta, \beta']|}}.
\end{align*}
Now we take $E:=\pm\frac{1}{4\lambda \sqrt[4]{|\lambda|}}$ (depending on the sign of $B$).
\par In order to prove the last part of the theorem let us define $C:=-\lambda f$, where
\begin{align*}
|\lambda|=\frac{1}{\sqrt[5]{4^4}\cdot\sqrt[5]{E^4}}.
\end{align*}
We also denote $\gamma_1:=(\alpha,\alpha)$ and $\gamma_2:=(\beta,-\beta)$. Now, by Lemma \ref{lm::DetRelation} and using the fact that $\gamma_1, \gamma_1',\gamma_1''$ are linearly dependent (in consequence $\det[\gamma_1,\gamma_1'',\gamma_2,\gamma_1']=0$) we get
\begin{align*}
\det[f_x,f_y,f_z,f]=\det [g_x,g_y,g,\Jt g]=-8B\det[\alpha ,\alpha']\cdot\det[\beta,\beta']\neq 0.
\end{align*}
The above implies that $f$ is an immersion and $C$ is transversal.
Of course $C$ is $\Jt$-tangent as well.
We also have
\begin{align*}
|\Theta  (\frac{\partial}{\partial x},\frac{\partial}{\partial y},\frac{\partial}{\partial z})|&=|\det [f_x,f_y,f_z,-\lambda f]|\\
&=8|\lambda |\cdot |B|\cdot|\det [\alpha,\alpha']\cdot\det [\beta, \beta']|\\
&=\frac{2\sqrt{|\det [\alpha,\alpha']\cdot\det [\beta, \beta']|}}{\sqrt[4]{|\lambda|}}
\end{align*}
Directly from (\ref{eq::formf}) we obtain
\begin{align*}
f_x=(\alpha,\alpha)(y)(\cosh z-\sinh z),
\end{align*}
so in particular $\frac{\partial}{\partial x}\in \DD^+$ because $\Jt f_x=f_x$. Moreover $f_{xx}=0$ implies $h(\frac{\partial}{\partial x},\frac{\partial}{\partial x})=0$ so $\DD^+$ is the null-direction for $f$. Finally, using the Gauss formula we compute
\begin{align*}
\omega_h(\frac{\partial}{\partial x},\frac{\partial}{\partial y},\frac{\partial}{\partial z})&=\frac{1}{\sqrt{|\lambda|}}\cdot |h(\frac{\partial}{\partial x},\frac{\partial}{\partial y})|=\frac{1}{\sqrt{|\lambda|}}\cdot \frac{1}{2|\lambda|\cdot |B|}\\
&=\frac{4|\lambda|\sqrt[4]{|\lambda|}\cdot\sqrt{|\det [\alpha,\alpha']\cdot\det [\beta, \beta']|} }{2|\lambda|\sqrt{|\lambda|}}
\end{align*}
That is $\omega_h=|\Theta|$.
\end{proof}
%%%%%%%%%%%%%%%%%%%%%%%%%%%%%%
%%%%%%%%%%%%%%%%%%%%%%%%%%%%%%

In a similar way one may prove
%%%%%%%%%%%%%%%%
%%    THEOREM  %
%%%%%%%%%%%%%%%%
\begin{thm}
Let $f\colon M\mapsto \R^4$ be an affine hypersphere with the null-direction $\DD^-$ and the $\Jt$-tangent Blaschke field $C$. Then there exist smooth planar curves
$\alpha\colon I\rightarrow \R^2$, $\beta\colon I\rightarrow \R^2$ with the property
\begin{align}\label{eq::detalphabetaDminus}
\det [\alpha,\alpha']\cdot \det [\beta,\beta']\neq 0
\end{align}
and a smooth function $A\colon I\rightarrow\R$, where $I$ is an open interval in $\R$, such that $f$ can be locally expressed in the form
\begin{align}\label{eq::formfDminus}
f(x,y,z)=\Jt g(x,y)\cosh z-g(x,y)\sinh z,
\end{align}
where $g(x,y)=(y+A(x))\cdot (\alpha,-\alpha)(x)+B(x)(\alpha',-\alpha')(x)+(\beta,\beta)(x)$ and
$$
B(x)=\frac{E}{\sqrt{|\det [\alpha(x),\alpha'(x)]\cdot \det [\beta(x),\beta'(x)]|}}
$$
for some non-zero constant $E$.
\par Moreover, for any smooth curves $\alpha, \beta\colon I\rightarrow \R^2$ such that (\ref{eq::detalphabetaDminus}) and for any smooth function $A\colon I\rightarrow \R$ and a non-zero constant $E$  the function $f$ given by {\rm(\ref{eq::formfDminus})} is the $\Jt$-tangent affine hypersphere with the null-direction $\DD^-$.

\end{thm}
\begin{exa}
Let us consider two affine immersions ($i=1,2$) defined by
$$
f_i\colon \R^3\ni (x,y,z)\mapsto \Jt g_i(x,y)\cosh z-g_i(x,y)\sinh z\in\R^4
$$
where
\begin{align*}
g_1\colon \R^2\ni (x,y)\mapsto \left [\begin{matrix}
x+\cos y\\
yx+\sin y+\frac{1}{4}\\
x-\cos y\\
yx-\sin y+\frac{1}{4}
\end{matrix}\right ]\in \R^4
\end{align*}
and
\begin{align*}
g_2\colon \R^2\ni (x,y)\mapsto \left [\begin{matrix}
y+\cos x\\
xy+\sin x+\frac{1}{4}\\
-y+\cos x\\
-xy+\sin x-\frac{1}{4}
\end{matrix}\right ]\in \R^4
\end{align*}
with the transversal vector fields $C_i=-f_i$ for $i=1,2$. Of course $C_i$ is $\Jt$-tangent and $\tau_i =0$ and $S_i=\id$ for $i=1,2$.
By straightforward computations we obtain
\begin{align*}
h_1= \left [\begin{matrix}
0 & -2 & 0\\
-2 & \frac{1}{2} & 4x\\
0 & 4x &-1
\end{matrix}\right ],
\qquad
h_2= \left [\begin{matrix}
\frac{1}{2} & -2 & -4y\\
-2 & 0 & 0\\
-4y & 0 &-1
\end{matrix}\right ]
\end{align*}
in the canonical basis $\{\frac{\partial}{\partial x},\frac{\partial}{\partial y},\frac{\partial}{\partial z}\}$ of $\R^3$. Since $\det h_i=4$, $f_i$ is nondegenerate. In particular $\omega_{h_i}(\frac{\partial}{\partial x},\frac{\partial}{\partial y},\frac{\partial}{\partial z})=2$. By straightforward computations we also get
\begin{align*}
\Theta_i (\frac{\partial}{\partial x},\frac{\partial}{\partial y},\frac{\partial}{\partial z})=\det [{f_i}_x,{f_i}_y,{f_i}_z,C_i]=2
\end{align*}
So $f_i$ is the affine hypersphere for $i=1,2$.
\par For $f_1$ we have ${f_1}_x=\Jt {f_1}_x$ so $\frac{\partial}{\partial x}\in \DD^+$. Since $h_1(\frac{\partial}{\partial x},\frac{\partial}{\partial x})=0$ we have that $\DD^+$ is the null-direction for $f_1$. We also have
\begin{align*}
\Jt (x^2{f_1}_x+\frac{1}{4}{f_1}_y+x{f_1}_z)=-(x^2{f_1}_x+\frac{1}{4}{f_1}_y+x{f_1}_z)
\end{align*}
so the vector field $Y:=x^2\frac{\partial}{\partial x}+\frac{1}{4}\frac{\partial}{\partial y}+x\frac{\partial}{\partial z}$ belongs to $\DD^-$. Now, we compute
\begin{align*}
h_1(Y,Y)=2\cdot\frac{1}{4}x^2\cdot h_1(\frac{\partial}{\partial x},\frac{\partial}{\partial y})+\frac{1}{16}\cdot h_1(\frac{\partial}{\partial y},\frac{\partial}{\partial y})\\
+2\cdot\frac{1}{4}x\cdot h_1(\frac{\partial}{\partial y},\frac{\partial}{\partial z})+x^2\cdot h_1(\frac{\partial}{\partial z},\frac{\partial}{\partial z})=\frac{1}{32}\neq 0
\end{align*}
so $\DD^-$ is not the null-direction for $f_1$.
\par On the other hand, for $f_2$ we have ${f_2}_y=-\Jt {f_2}_y$ so $\frac{\partial}{\partial y}\in \DD^-$. Since $h_2(\frac{\partial}{\partial y},\frac{\partial}{\partial y})=0$ we have that $\DD^-$ is the null-direction for $f_2$.
We also have
\begin{align*}
\Jt (\frac{1}{4}{f_2}_x+y^2{f_2}_y-y{f_2}_z)=\frac{1}{4}{f_2}_x+y^2{f_2}_y-y{f_2}_z
\end{align*}
so the vector field $X:=\frac{1}{4}\frac{\partial}{\partial x}+y^2\frac{\partial}{\partial y}-y\frac{\partial}{\partial z}$ belongs to $\DD^+$. Now, we compute
\begin{align*}
h_2(X,X)=\frac{1}{16}\cdot h_2(\frac{\partial}{\partial x},\frac{\partial}{\partial x})+2\cdot\frac{1}{4}y^2\cdot h_2(\frac{\partial}{\partial x},\frac{\partial}{\partial y})\\
+2\cdot\frac{1}{4}(-y)\cdot h_2(\frac{\partial}{\partial x},\frac{\partial}{\partial z})+y^2\cdot h_2(\frac{\partial}{\partial z},\frac{\partial}{\partial z})=\frac{1}{32}\neq 0
\end{align*}so $\DD^+$ is not the null-direction for $f_2$.
\par Summarising $f_1$ is a $\Jt$-tangent affine hypersphere with the null-direction $\DD^+$ (and not $\DD^-$) while
$f_2$ is a $\Jt$-tangent affine hypersphere with the null-direction $\DD^-$ (and not $\DD^+$). Both hypersurfaces are not hyperquadrics.
\end{exa}

\emph{This Research was financed by the Ministry of Science and Higher Education of the Republic of Poland.}


\begin{thebibliography}{HD}

%% Use the widest label as the parameter.

%% In IMPAN journals, only the title is italicized; boldface is not used.
%% The issue number is only given when the issues are paginated separately.

%%%%%%% To ease editing, use normal size:

\normalsize
\baselineskip=17pt

%%%%%%%%%%%%%%%

\bibitem{Z}
{S. Zamkovoy},
{\emph{Canonical connections on paracontact manifolds}},
{Ann. Glob. Anal. Geom. \textbf{36} (2009)},
{37–-60}.

\bibitem{CKM}
{B. Cappelletti-Montano, I. K\"upeli Erken, C. Murathan},
{\emph{Nullity conditions in paracontact geometry}},
{Diff. Geom. Appl. \textbf{30} (2012)},
{665--693}.

\bibitem{KM}
{I. K\"upeli Erken, C. Murathan},
{\emph{A Complete Study of Three-Dimensional Paracontact ($\tilde{k}, \tilde{\mu}, \tilde{v}$)-spaces}},
{Int. J. Geom. Methods. Mod. Phys. \textbf{14}(7) (2017)},
{DOI: 10.1142:/S0219887817501067}.

\bibitem{NS}
{K. Nomizu and T. Sasaki},
{\emph{Affine differential geometry}},
{Cambridge University Press},
{1994}.

%\bibitem{DVV}
%{F. Dillen, L. Vrancken, L. Verstraelen},
%{\emph{Complex affine differential geometry}},
%{Atti Acc. Peloritana dei Pericolanti LXVI (1988)}, {232--260}.

\bibitem{SZ}
{Z. Szancer},
{\emph{On para-complex affine hyperspheres}},
{Results Math. \textbf{72} (2017)},
{491-–513}.

\bibitem{SZZ}
{Z. Szancer},
{\emph{$\widetilde{J}$-tangent affine hypersurfaces with an induced almost paracontact structure}},
{arXiv:1710.10488}.

%\bibitem{SZ2}
%{Z. Szancer},
%{\emph{J-tangent affine hyperspheres with an involutive contact distribution}}
%{Publ. Math. Debrecen \textbf{89/4} (2016)},
%{399--413}.

\bibitem{SZ3}
{Z. Szancer},
{\emph{On $\Jt$-tangent affine hyperspheres}},
{arXiv:1804.01599}.

\bibitem{CLS}
{ V. Cort\'es, M. A. Lawn and L. Sch\"afer},
{\emph{ Affine hyperspheres associated to special para-K\"ahler manifolds}},
{ Int. J. Geom. Methods Mod. Phys. \textbf{3} (2006)},
{995--1009}.

\bibitem{CMMS}
{V. Cort\'es, C. Mayer, T. Mohaupt, F. Saueressing},
{\emph{Special geometry of Euclidean supersymmetry I: Vector multiplets}},
{J. High Energy Phys. \textbf{03} (2004)},
{028}.

\bibitem{LS}
{M. A. Lawn and L. Sch\"afer},
{\emph{Decompositions of para-complex vector bundles and para-complex affine
immersions}},
{Results Math. \textbf{48} (2005)},
{246--274}.

%\bibitem{AB}
%{A. Al-Aqeel, A. Bejancu},
%{\emph{On the geometry of paracomplex submanifolds}},
%{Demonstratio Math. \textbf{34} (2001) No. 4}, {919--932}.

\bibitem{CFG}
{V. Cruceanu, P. Fortuny, P. M. Gadea,}
{\emph{A survey on para-complex geometry}},
{Rocky Mountain J. Math. \textbf{26} (1996) 1}, {83--115}.

\bibitem{KW}
{S. Kaneyuki, F. L. Williams},
{\emph{Almost paracontact and parahodge structures on manifolds}},
{Nagoya Math. J. \textbf{99} (1985)}, {173--187}.



\end{thebibliography}
\end{document}